\documentclass[english]{article}
\usepackage[T1]{fontenc}
\usepackage[latin9]{inputenc}
\usepackage{geometry}
\usepackage{color}
\usepackage{mathrsfs}
\usepackage{amsmath}
\usepackage{amsthm}
\usepackage{amssymb}
\usepackage{stmaryrd}
\usepackage{graphicx}

\makeatletter

\providecommand{\tabularnewline}{\\}
\newcommand{\lyxdot}{.}

\numberwithin{equation}{section}
\numberwithin{figure}{section}
\theoremstyle{plain}
\newtheorem{thm}{\protect\theoremname}
\theoremstyle{definition}
\newtheorem{defn}[thm]{\protect\definitionname}
\theoremstyle{plain}
\newtheorem{fact}[thm]{\protect\factname}
\theoremstyle{remark}
\newtheorem{rem}[thm]{\protect\remarkname}
\theoremstyle{plain}
\newtheorem{prop}[thm]{\protect\propositionname}
\theoremstyle{plain}
\newtheorem{cor}[thm]{\protect\corollaryname}

\date{}

\makeatother

\usepackage{babel}
\providecommand{\corollaryname}{Corollary}
\providecommand{\definitionname}{Definition}
\providecommand{\factname}{Fact}
\providecommand{\propositionname}{Proposition}
\providecommand{\remarkname}{Remark}
\providecommand{\theoremname}{Theorem}

\begin{document}
\global\long\def\pplus{\overset{{\scriptscriptstyle \perp}}{\oplus}}%
\global\long\def\oton#1#2{#1_{1},\dots,#1_{#2}}%
\global\long\def\rk{\text{rk}}%
\global\long\def\tr{\text{tr}}%
\global\long\def\adj{\text{adj}}%
\global\long\def\im{\text{Im}}%
\global\long\def\span{\text{span}}%
\global\long\def\nin{\notin}%
\global\long\def\id{\text{Id}}%
\global\long\def\makbil{\text{parallel}}%
\global\long\def\meridian{\text{meridian}}%
\global\long\def\sgn{\text{Sgn}}%
\global\long\def\fix{\text{Fix}}%
\global\long\def\hess{\text{Hess}}%
\global\long\def\ooton#1#2{#1_{0},\dots,#1_{#2}}%
\global\long\def\otherwise{\text{Otherwise}}%
\global\long\def\norm#1{\left\Vert #1\right\Vert }%
\global\long\def\nil{\text{Nill}}%
\global\long\def\fix{\text{Fix}}%
\global\long\def\spec{\text{Spec}}%
\global\long\def\ind{\text{Ind}}%
 
\global\long\def\hom{\text{Hom}}%
\global\long\def\ob{\text{Ob}}%
\global\long\def\coker{\text{coker}}%
\global\long\def\rad{\text{Rad}}%
\global\long\def\supp{\text{Supp}}%
\global\long\def\aut{\text{Aut}}%
\global\long\def\gal{\text{Gal}}%
\global\long\def\ann{\text{Ann}}%
\global\long\def\mayervi{\text{Mayer-Vietoris}}%
\global\long\def\conv{\text{conv}}%
\global\long\def\diam{\text{diam}}%
\global\long\def\length{\text{length}}%
\global\long\def\tp{\text{tp}}%
\global\long\def\lcm{\text{lcm}}%
\global\long\def\core{\text{Core}}%
\global\long\def\ad{\text{ad}}%
\global\long\def\ord{\text{ord}}%
\global\long\def\rank{\text{rank}}%
\global\long\def\Inn{\text{Inn}}%
\global\long\def\involution{\overline{{\scriptstyle \bigbox}}}%

\title{Classes of free group extensions.}
\author{Noam M.D. Kolodner}
\maketitle
\begin{abstract}
In this paper we identify different classes of free group extension
using core graphs, by further developing machinery from \cite{KOLODNER2020}.
We show that every free group extension $H\leq K\leq F$ has a base
$B$ such that the associated pointed graph morphism $\Gamma_{B}\left(H\right)\to\Gamma_{B}\left(H\right)$
is onto. But if we examine graphs without base points, there is an
extension $\left\langle b\right\rangle \leq\left\langle b,aba^{-1}\right\rangle <F_{\left\{ a,b\right\} }$
such that for every base of $F_{\left\{ a,b\right\} }$ the associated
graph morphisms are injective.
\end{abstract}

\section{Introduction}

In this paper we identify different classes of free group extension
using core graphs, by further developing machinery from \cite{KOLODNER2020}.
Leveraging the theory of topological cover spaces, Stallings \cite{MR695906}
established a correspondence between subgroups of a free group with
labeled graphs called core graphs. Let $F$ be a free group and $B$
a base. For every subgroup $H\leq F$ we associate a pointed labeled
graph $\Gamma_{B}\left(H\right)$ and for subgroups $H\leq K\leq F$
a graph morphism $\Gamma_{B}\left(H\right)\to\Gamma_{B}\left(K\right)$.
Thus we realize the category of subgroups of free groups ordered by
inclusion as a subcategory of the category of pointed labeled graphs. When
we order subgroups by inclusion,\textcolor{red}{{} }these inclusion
have no ``flavor''. But as morphisms in the category of labeled
graphs they have easily detectable properties. Labeled graph morphisms
can be injective or surjective for instance. The problem is that these
properties are incidental and are dependent on the arbitrary base
chosen for constructing the correspondence, so if one wants to leverage
these properties to study algebraic properties of free group extensions
one must look at invariant properties.

The first attempt to do this was made by Miasnikov Ventura and Weil
\cite{MR2395796}, who conjectured that extensions $H\leq K\leq F$,
such that the morphism of the corresponding core graphs is surjective
for every base, are algebraic extensions (i.e. such that $H$ is not
included in any proper free factor of $K$). Puder and Parzanchevski
showed this to be false for subgroups of a free group with two generators
but conjectured that it is still true for free groups with more generators,
or that it is true if one allows automorphisms of free extension of
the ambient free group $F$. The author of this paper found a counter
example for the revised conjectures \cite{KOLODNER2020} and thus
showed that algebraic extensions are strictly included in extensions
that are onto on all bases, which form a separate extension class.
Parzanchevski and Puder \cite{MR3264763} suggested another class
of extensions: one where there \emph{exists} a base for which the
graph morphism is onto. They asked if this is true for all extensions.
Another suggestion proposed by Verd\'{u} in his masters thesis \cite{meunpub2020}
is an extension class where graph morphisms are injective for every
base.

In Theorem \ref{thm:ontoallbase} we show that for every extension
$H\leq K\leq F$ there exists a base such that the morphism of the
corresponding labeled graphs is onto; moreover one obtains this base
by conjugation. Thus the extension class where there exists a base
that is onto includes all extensions and the extension class where
every base is injective is empty. If instead of the category of pointed
labeled graphs we look at graphs without base points, and instead
of automorphisms we look at outer automorphisms we show that in this
setting these classes of extensions are non-trivial. Using the methods
developed in \cite{KOLODNER2020} we show that the extension $\left\langle b\right\rangle \leq\left\langle b,aba^{-1}\right\rangle $
is injective for every outer automorphism. Thus in the new setting
there is a class of extension whose graph morphism is injective on
every base and there is also a class of extensions s.t. there exists
a base where the graph morphism is surjective. In order to prove that
extension $\left\langle b\right\rangle \leq\left\langle b,aba^{-1}\right\rangle $
has the desired property we further develop in this paper the machinery
from \cite{KOLODNER2020}, to deal with cases where a graph does not
have stencil finiteness.

\section{\label{sec:Preliminaries}Preliminaries}

In this paper we use the machinery developed in \cite{KOLODNER2020}.
We present the definitions we use in the paper, merging the language
of \cite{MR1882114} and \cite{MR695906}.
\begin{defn}[Graphs]
We use graphs in the sense of Serre \cite{MR0476875}: A \emph{graph}
$\Gamma$ is a set $V\left(\Gamma\right)$ of vertices and a set $E\left(\Gamma\right)$
of edges with a function $\iota\colon E\left(\Gamma\right)\to V\left(\Gamma\right)$
called the initial vertex map and an involution $\involution\colon E\left(\Gamma\right)\to E\left(\Gamma\right)$
with $\overline{e}\neq e$ and $\overline{\overline{e}}=e$. A \emph{graph
morphism} $f\colon\Gamma\to\Delta$ is a pair of set functions $f^{E}\colon E\left(\Gamma\right)\to E\left(\Delta\right)$
and $f^{V}\colon V\left(\Gamma\right)\to V\left(\Delta\right)$ that
commute with the structure functions. A \emph{path }in $\Gamma$ is
a finite sequence $\oton en\in E\left(\Gamma\right)$ with $\iota\left(\overline{e}_{k}\right)=\iota\left(e_{k+1}\right)$
for every $1\leq k<n$. The path is \emph{closed or circuit} if $\iota\left(\overline{e}_{n}\right)=\iota\left(e_{1}\right)$,
and \emph{reduced} if $e_{k+1}\neq\overline{e}_{k}$ for all $k$.
All the graphs in the paper are assumed to be connected unless specified
otherwise, namely, for every $e,f\in E\left(\Gamma\right)$ there
is a path $e,\oton en,f$.
\end{defn}

\begin{defn}[Labeled graphs]
Let $X$ be a set and let $X^{-1}$ be the set of its formal inverses.
We define $R_{X}$ to be the graph with $E\left(R_{X}\right)=X\cup X^{-1}$,
$V\left(R_{X}\right)=\left\{ *\right\} $, $\overline{x}=x^{-1}$
and $\iota\left(x\right)=*.$ An \emph{$X$-labeled graph }is a graph
$\Gamma$ together with a graph morphism $l\colon\Gamma\to R_{X}$.
A morphism of $X$-labeled graphs $\Gamma$ and $\Delta$ is a graph
morphism $f\colon\Gamma\to\Delta$ that commutes with the label functions.
Let $\mathcal{P}\left(\Gamma\right)$ be the set of all the paths
in $\Gamma$, let $F_{X}$ be the free group on $X$ and let $P=\oton en$
be a path. The edge part of the label function $l^{E}\colon E\left(\Gamma\right)\to E\left(R_{X}\right)$
can be extended to a function $l\colon\mathcal{P}\left(\Gamma\right)\to F_{X}$
by the rule $l\left(P\right)=l\left(e_{1}\right)\ldots\left(e_{n}\right)$.

A \emph{pointed }$X$-labeled graph is an $X$-labeled graph that
has a distinguished vertex called the base point. A morphism of a
pointed labeled graph sends the base point to the base point. This
constitutes a category called $X\text{-Grph}$. For a pointed $X$-labeled
graph $\Gamma$ we define $\pi_{1}\left(\Gamma\right)$ to be 
\[
\pi_{1}(\Gamma)=\left\{ l\left(P\right)\in F_{X}\,|\,P\text{ is a closed path beginning at the base point}\right\} .
\]
\end{defn}

\begin{defn}
Let $F_{X}$ be the free group on the set $X$. We define a category
$\text{Sub}\left(F_{X}\right)$, whose objects are subgroups $H\leq F_{X}$
and there is a unique morphism in $\hom\left(H,K\right)$ iff $H\leq K$.
It is easy to verify that $\pi_{1}$ is a functor from $X\text{-Grph}$
to $\text{Sub}\left(F_{X}\right)$.
\end{defn}

Note: The functor $\pi_{1}$ defined above is not the fundamental
group of $\Gamma$ as a topological space. Rather, if one views $\Gamma$
and $R_{X}$ as topological spaces and $l$ as a continuous function,
then $\pi_{1}$ is the image of the fundamental group of $\Gamma$
in that of $R_{X}$ under the group homomorphism induced by $l$.
\begin{defn}[Folding]
A labeled graph $\Gamma$ is \emph{folded }if $l\left(e\right)\neq l\left(f\right)$
holds for every two edges $e,f\in E\left(\Gamma\right)$ with $\iota\left(f\right)=\iota\left(e\right)$.
We notice that there is at most one morphism between two pointed folded
labeled graphs. If $\Gamma$ is not folded, there exist $e,f\in E\left(\Gamma\right)$
s.t.\ $\iota\left(e\right)=\iota\left(f\right)$ and $l\left(e\right)=l\left(f\right)$;
Let $\Gamma'$ be the graph obtained by identifying the vertex $\iota\left(\overline{e}\right)$
with $\iota\left(\overline{f}\right)$, the edges $e$ with $f$ and
$\overline{e}$ with $\overline{f}$. We say $\Gamma'$ is the result
of \emph{folding} $e$ and $f$. The label function $l$ factors through
$\Gamma'$, yielding a label function $l'$ on $\Gamma'$, and we
notice that $\pi_{1}\left(\Gamma\right)=\pi_{1}\left(\Gamma'\right)$.
\end{defn}

\begin{defn}[Core graph]
A \emph{core graph} $\Gamma$ is a labeled, folded, pointed graph
s.t.\ every edge in $\Gamma$ belongs in a closed reduced path around
the base point. Is case of finite graphs this is equivalent to every
$v\in V\left(\Gamma\right)$ having $\deg(v):=|\iota^{-1}(v)|>1$
except the base point which can have any degree.
\end{defn}

\begin{defn}
Let $X\text{-CGrph}$ be the category of connected, pointed, folded,
$X$-labeled core graphs. Define a functor $\Gamma_{X}\colon\text{Sub}\left(F_{X}\right)\to X\text{-CGrph}$
that associates to the subgroup $H\leq F_{X}$ a graph $\Gamma_{X}\left(H\right)$
(which is unique up to a unique isomorphism) s.t.\ $\pi_{1}(\Gamma_{X}(H))=H$.
\end{defn}

\begin{fact}[\cite{MR695906,MR1882114}]
The functors $\pi_{1}$ and $\Gamma$ define an equivalence between
the categories $X\text{-CGrph}$ and $\text{Sub}\left(F_{X}\right)$.
\end{fact}

The correspondence between the categories of $X\text{-CGrph}$ and
$\text{Sub}\left(F_{X}\right)$ follows from the theory of cover spaces.
Let us sketch a proof. We regard $R_{X}$ as a topological space and
look at the category of connected pointed cover spaces of $R_{X}$.
This category is equivalent to $\text{Sub}\left(F_{X}\right)$, following
from the fact that $R_{X}$ has a universal cover. Let $\Gamma$ be
a connected folded $X$-labeled core graph, viewed as a topological
space and $l$ as a continuous function. There is a unique way up
to cover isomorphism to extend $\Gamma$ to a cover of $R_{X}$. There
is also a unique way to associate a core graph to a cover space of
$R_{X}$. This gives us an equivalence between the category of connected
pointed cover spaces of $R_{X}$ and pointed connected folded $X$-labeled
core graphs.
\begin{defn}
By uniqueness of the core graph of a subgroup we can define a functor
$\core\colon\text{\ensuremath{X}-Grph}\to X\text{-CGrph}$ that associates
to a graph $\Gamma$ a core graph s.t.\ $\pi_{1}\left(\core\left(\Gamma\right)\right)=\pi_{1}\left(\Gamma\right)$.
\end{defn}

\begin{defn}
Let $\Gamma$ be a graph with a vertex $v$ of degree one which is
not the base-point. For $e=\iota^{-1}\left(v\right)$, let $\Gamma'$
be the graph with $V\left(\Gamma'\right)=V\left(\Gamma\right)\backslash\left\{ v\right\} $
and $E\left(\Gamma'\right)=E\left(\Gamma\right)\backslash\left\{ e,\overline{e}\right\} $.
We say that $\Gamma'$ is the result of \emph{trimming} $e$ from
$\Gamma$, and we notice that $\pi_{1}\left(\Gamma\right)=\pi_{1}\left(\Gamma'\right)$.
\end{defn}

\begin{rem}
For a finite graph $\Gamma$, after both trimming and folding $\left|E\left(\Gamma'\right)\right|<\left|E\left(\Gamma\right)\right|$.
If no foldings or trimmings are possible then $\Gamma$ is a core
graph. This means that after preforming a finite amount of trimmings
and foldings we arrive at $\core\left(\Gamma\right)$. It follows
from the uniqueness of $\core\left(\Gamma\right)$ that the order
in which one performs the trimmings and foldings does not matter.
\end{rem}

\begin{defn}[Whitehead graph]
A \emph{2-path} in a graph $\Gamma$ is a pair $\left(e,f\right)\in E\left(\Gamma\right)\times E\left(\Gamma\right)$
with $\iota\left(f\right)=\iota\left(\overline{e}\right)$ and $f\neq\overline{e}$.
If $\Gamma$ is $X$-labeled, the set 
\[
W\left(\Gamma\right)=\left\{ \left\{ l\left(e\right),l\left(\overline{f}\right)\right\} \mid\left(e,f\right)\text{ is a \ensuremath{2}-path in \ensuremath{\Gamma}}\right\} 
\]
forms the set of edges of a combinatorial (undirected) graph whose
vertices are $X\cup X^{-1}$, called the \emph{Whitehead graph }of
$\Gamma$. If $w\in F_{X}$ is a cyclically reduced word, the Whitehead
graph of $w$ as defined in \cite{MR1575455,MR1714852} and the Whitehead
graph of $\Gamma\left(\left\langle w\right\rangle \right)$ defined
here coincide. Let $W_{X}=W\left(R_{X}\right)$ be the set of edges
of the Whitehead graph of $R_{X}$, which we call the \emph{full Whitehead
graph}. Let $x,y\in X\cup X^{-1}$ and let $\left\{ x,y\right\} \in W_{X}$
be an edge. We denote $x.y=\left\{ x,y\right\} $ (this is similar
to the notation in \cite{MR1812024}).
\end{defn}

\begin{defn}
A homomorphism $\varphi\colon F_{Y}\to F_{X}$ is \emph{non-degenerate}
if $\varphi\left(y\right)\neq1$ for every $y\in Y$.
\end{defn}

\begin{defn}
Let $w\in F_{X}$ be a reduced word of length $n$. Define $\Gamma_{w}$
to be the $X$-labeled graph with $V\left(\Gamma_{w}\right)=\left\{ 1,\dots,n+1\right\} $
forming a path $P$ labeled by $l\left(P\right)=w$. Notice that $\Gamma_{w}\cong\Gamma_{w^{-1}}$.
\end{defn}

\begin{defn}
Let $\varphi\colon F_{Y}\to F_{X}$ be a non-degenerate homomorphism.
We define a functor $\mathcal{F}_{\varphi}$ from $Y$-labeled graphs
to $X$-labeled graphs by sending $y$-labeled edges to $\varphi\left(y\right)$-labeled
paths. Formally, let $\Delta$ be a $Y$-labeled graph and let $E_{0}=\left\{ e\in E\left(\Delta\right)|l\left(e\right)\in Y\right\} $
be an orientation of $\Delta$, namely, $E\left(\Delta\right)=E_{0}\sqcup\{\overline{e}\,|\,e\in E_{0}\}$.
For every $e\in E\left(\Delta\right)$ let $n_{e}\in\mathbb{N}$ be
the length of the word $\varphi\left(l\left(e\right)\right)\in F_{X}$
plus one. We consider $V\left(\Delta\right)$ as a graph without edges,
take the disjoint union of graphs $\bigsqcup_{e\in E_{0}}\Gamma_{\varphi\left(l\left(e\right)\right)}\sqcup V\left(\Delta\right)$
and for every $e\in E_{0}$ glue $1\in V\left(\Gamma_{\varphi\left(l\left(e\right)\right)}\right)$\textcolor{red}{{}
}to $\iota\left(e\right)\in V\left(\Delta\right)$, and $n_{e}\in V\left(\Gamma_{\varphi\left(l\left(e\right)\right)}\right)$
to $\iota\left(\overline{e}\right)\in V\left(\Delta\right)$. Let
$\Delta$ and $\Xi$ be $Y$-labeled graphs, and $f\colon\Delta\to\Xi$
a graph morphism. As for functionality, if $f\colon\Delta\to\Xi$
is a morphism of $Y$-labeled graphs, $\mathcal{F}_{\varphi}f$ is
defined as follows: every edge in $\mathcal{F}_{\varphi}\left(\Delta\right)$
belongs to a path $\mathcal{F}_{\varphi}\left(e\right)$ for some
$e\in E(\Delta)$, and we define $\left(\mathcal{F}_{\varphi}f\right)\left(\mathcal{F}_{\varphi}\left(e\right)\right)=\mathcal{F}_{\varphi}\left(f\left(e\right)\right)$.
\end{defn}

\begin{rem}
\label{rem:Let--we}For $H\leq F_{Y}$ we notice that $\core(\mathcal{F}_{\varphi}\Gamma_{Y}(H))=\Gamma_{X}\left(\varphi\left(H\right)\right)$.
\end{rem}

\begin{defn}[Stencil]
Let $\Gamma$ be an $Y$-labeled graph, and $\varphi\colon F_{Y}\to F_{X}$
a non-degenerate homomorphism. We say that the pair $\left(\varphi,\Gamma\right)$
is a \textit{stencil} iff $\mathcal{F}_{\varphi}\left(\Gamma\right)$
is a folded graph. Notice that if $\Gamma$ is not folded, then $\mathcal{F}_{\varphi}\left(\Gamma\right)$
is not folded for any $\varphi$.
\end{defn}

\begin{defn}
Let $\Gamma$ be a $Y$-labeled folded graph. An FGR object $(Y,N_{Y})$
is said to be a \textit{stencil space}\emph{ of }$\Gamma$ if $W\left(\Gamma\right)\subseteq N_{Y}$.
The reason for the name is that for any object $\left(X,N_{X}\right)$
and morphism $\varphi\in\hom\left((Y,N_{Y}),\left(X,N_{X}\right)\right)$,
the pair $\left(\varphi,\Gamma\right)$ is a stencil.
\end{defn}

\begin{defn}
Let $\tau\colon F_{X}\backslash\left\{ 1\right\} \to X\cup X^{-1}$
be the function returning the last letter of a reduced word. For reduced
words $u,v$ in a free group, we write $u\cdot v$ to indicate that
there is no cancellation in their concatenation, namely $\tau(u)\neq\tau(v^{-1})$.
\end{defn}

\begin{defn}[FGR]
The objects of the category \emph{Free Groups with Restrictions}
($\mathbf{FGR}$) are pairs $\left(X,N\right)$ where $X$ is a set
of ``generators'' and $N\subseteq W_{X}$ a set of ``restrictions''.
A morphism $\varphi\in\hom_{\mathbf{FGR}}\left(\left(X,N\right),\left(Y,M\right)\right)$
is a group homomorphism $\varphi\colon F_{Y}\rightarrow F_{X}$ with
the following properties:
\begin{enumerate}
\item[(i)]  For every $x\in Y$, $\varphi\left(x\right)\neq1$ ($\varphi$ is
non-degenerate).
\item[(ii)] For every $x\in Y$, $W(\Gamma_{\varphi\left(x\right)})\subseteq M$.
\item[(iii)] For every $x.y\in N$, $\varphi(x)\cdot\varphi(y)^{-1}$ (i.e.\ $\tau\left(\varphi\left(x\right)\right)\neq\tau\left(\varphi\left(y\right)\right)$).
\item[(iv)] For every $x.y\in N$, $\tau\left(\varphi\left(x\right)\right).\tau\left(\varphi\left(y\right)\right)\in M$.\footnote{Technically, (iv) implies (iii), as $M\subseteq W_{Y}$ and $x.x\notin W_{Y}$.}
\end{enumerate}
\end{defn}

\section{There is always a base where the graph morphism is onto}

All lemmas and propositions used here can be found in the section
titled The core functor in \cite{KOLODNER2020}.
\begin{prop}
\label{prop:Let--a}Let $H\leq F_{Y}$ a subgroup let $w\in F_{Y}$
be a word. We can obtain the core graph $\Gamma_{Y}\left(wHw^{-1}\right)$
from the graph $\Gamma_{Y}\left(H\right)$ by the following process: 
\begin{enumerate}
\item Attach a reduced path labeled $w$ to the base point in $\Gamma\left(H\right)$
\item Set the new base point to be at the beginning of the path labeled
by $w$. 
\item Fold and trim if necessary
\end{enumerate}
\end{prop}

\begin{thm}
\label{thm:ontoallbase}Let $F_{Y}$ be a free group with a finite
set of generators $Y$ and let $H\leq K\leq F_{Y}$ be finitely generated
subgroups. Then there is a basis $B$ of $F_{Y}$ s.t. the morphism
of Stallings graphs $\Gamma_{B}\left(H\right)\to\Gamma_{B}\left(K\right)$
is onto. Moreover this basis can be obtained by conjugation.
\end{thm}

\begin{proof}
We will prove the equivalent statement. There is an automorphism $\varphi\in\Inn\left(F_{Y}\right)$
s.t. $\Gamma_{Y}\left(\varphi\left(H\right)\right)\to\Gamma_{Y}\left(\varphi\left(K\right)\right)$
is onto. Without loss of generality we can assume that the base point
of $\Gamma_{Y}\left(H\right)$ has degree greater or equal to two
(if not this can be corrected by conjugation). Let $u\in F_{Y}$ be
the label of a reduced circuit in $\Gamma_{Y}\left(K\right)$ based
at the base point that traverses each edge of $\Gamma_{Y}\left(H\right)$
at least once (we are not bothered if it traverses some edges multiple
times). Let $\varphi\in\text{Inn}\left(F_{Y}\right)$ be conjugation
by $u$. By construction $u\in K$ therefore $\core\mathcal{F}_{\varphi}\left(\Gamma_{Y}\left(K\right)\right)=\Gamma_{Y}\left(\varphi\left(K\right)\right)=\Gamma_{Y}\left(uKu^{-1}\right)=\Gamma_{Y}\left(K\right)$.
We construct $\Gamma_{Y}\left(uHu^{-1}\right)$ by the process described
in proposition \ref{prop:Let--a} but stop at stage 2 (before preforming
folding and trimming), we denote this graph by $\Gamma'.$ By construction
the graph morphism $\Gamma'\to\Gamma_{Y}\left(K\right)$ is onto.
Because $\Gamma_{Y}\left(H\right)$ is folded and its base point has
degree at least 2 the graph $\Gamma'$ satisfies the conditions of
lemma 2.1 from \cite{MR3211804} this means that $\core\left(\Gamma'\right)$
is obtained without trimming. Without trimming the morphism remains
onto when one takes its core ( remark $3.16$ from \cite{KOLODNER2020}).
Thus we get that $\core\left(\Gamma'\to\Gamma_{Y}\left(K\right)\right)$
is onto but $\core\left(\Gamma'\to\Gamma_{Y}\left(K\right)\right)=\Gamma_{Y}\left(\varphi\left(H\right)\right)\to\Gamma_{Y}\left(\varphi\left(K\right)\right).$
\end{proof}
\begin{cor}
Let $H<K\leq F_{Y}$ ($H$ strictly contained in $K)$. There is a
basis $B$ of $F_{Y}$ s.t. the graph morphism $\Gamma_{B}\left(H\right)\to\Gamma_{B}\left(K\right)$
is not injective. 
\end{cor}

\begin{proof}
There exists a basis $B$ s.t. $\Gamma_{B}\left(H\right)\to\Gamma_{B}\left(K\right)$
is onto. Because $H$ is strictly contained in $K$ the graph morphism
$\Gamma_{B}\left(H\right)\to\Gamma_{B}\left(K\right)$ cannot be an
isomorphism therefore it is not injective. 
\end{proof}

\section{Example}

In the category of pointed labeled core graphs we can't have a free
group extension s.t. for every automorphism the graph morphism is
injective. We saw that the obstruction was conjugation we will show
that this is the only obstruction. For this we focus our attention
on labeled graphs without a base point. Miasnikov and Kapovich \cite{MR1882114}
called this `` the type of a graph''. To get the type of a graph
we forget the base point and trim again. We can look at the category
of core $X$-labeled graphs without base points. Let $\Gamma,\Delta$
be $X$-labeled core graphs. There is a graph morphism $\Gamma\to\Delta$
iff there exists a $u\in F_{X}$ s.t. $u\pi_{1}\left(\Gamma\right)u^{-1}<\pi_{1}\left(\Delta\right)$
(technically one has to choose a base point for $\pi_{1}$ to be well
defined. But it is well defined up to conjugation which is what we
are using here). We notice that without a base point there is no longer
a unique graph morphism between two graphs. Because we have an action
of $\text{Out}\left(F_{X}\right)$ on the set of subgroups of $F_{X}$
up to conjugation this gives us an action of $\text{Out}\left(F_{X}\right)$
on morphisms $\Gamma\to\Delta$ in the category of labeled graphs.
In this setting there is a graph morphism $\Gamma\to\Delta$ that
is injective in its whole orbit under outer automorphism. We give
the example of $\left\langle b\right\rangle <\left\langle b,aba^{-1}\right\rangle $.
We will use the tools developed in \cite{KOLODNER2020}. 
\begin{defn}
We define a functor 
\[
\text{Trimf}:PLCGraphs\to LCGraphs
\]
 from pointed labeled core graphs to labeled core graphs. The functor
forgets the base point and trims the ``tail'' (trimf stands for
forget then trim). It takes graph morphisms to their restrictions.
This definition is indeed a legal functor: Let $\Gamma\to\Delta$
be a morphism of pointed labeled core graphs. Let $v\in V\left(\Delta\right)$
be the base point of $\Delta$ and suppose it is of valency one (otherwise
no trimming occurs and the definition is clearly legal). Let $w\in V\left(\Gamma\right)$
be the inverse image of $v$ , it is the base point of $\Gamma$.
Since $\Gamma$ is folded the morphism $\Gamma\to\Delta$ is locally
injective therefor $w$ must also be of valency one. Let $\Gamma',\Delta'$
be the graphs obtained by trimming the edges incident to $v$ and
$w$ respectively. We see that the morphism $\Gamma'\to\Delta'$ obtained
by restricting $\Gamma\to\Delta$ to $\Gamma'$ is well defined. By
induction we can trim the whole ``tails'' of $\Gamma$ and $\Delta$.
(Trimf would not be defined if we include graphs that aren't folded).
\end{defn}

We denote $\Gamma=\Gamma_{\left\{ a,b\right\} }\left(\left\langle b\right\rangle \right)$
and $\Delta=\Gamma_{\left\{ a,b\right\} }\left(\left\langle b,aba^{-1}\right\rangle \right)$
and let $X$ be a countably infinite set. 
\begin{thm}
\label{thm:All-the-morphism}All the morphism in set $\left\{ \text{Trimf}\circ\core\circ\mathcal{F}_{\varphi}\left(\Gamma\to\Delta\right)|\varphi\in\hom\left(\left(\left\{ a,b\right\} ,\emptyset\right),\left(X,W_{X}\right)\right)\right\} $
are injective. 
\end{thm}

Theorem \label{thm:All-the-morphism-1} shows that for every morphism
$\Gamma\to\Delta$ in the orbit under $\text{Out}\left(F_{\left\{ a,b\right\} }\right)$
is injective: Let $u\in F_{X}$ and let $\varphi_{1},\varphi_{2}:F_{\left\{ a,b\right\} }\to F_{X}$
non-degenerate homomorphisms s.t. $u\varphi_{1}u^{-1}=\varphi_{2}$
then clearly $\text{Trimf}\circ\core\circ\mathcal{F}_{\varphi_{1}}\left(\Gamma\to\Delta\right)=\text{Trimf}\circ\core\circ\mathcal{F}_{\varphi_{2}}\left(\Gamma\to\Delta\right)$.
Without loss of generality we can assume that $\left\{ a,b\right\} \subset X$
so $\hom\left(\left(\left\{ a,b\right\} ,\emptyset\right),\left(X,W_{X}\right)\right)$
includes $\aut\left(F_{\left\{ a,b\right\} }\right)$ it includes
also all automorphisms of free extensions of $F_{\left\{ a,b\right\} }$
and non free extensions as well.  We will use the method presented
in \cite{KOLODNER2020} with modifications to account to the fact
that we are now interested in injective not surjective morphisms. 
\begin{rem}
\label{rem:injective}Let $\Gamma\to\Delta$ be a morphism of $U$-labeled
graphs and let $N_{U}$ be a set of restrictions. Suppose $\left(U,N_{U}\right)$
is a stencil space of $\Delta$.
\begin{enumerate}
\item $\left(U,N_{U}\right)$ is also a stencil space of $\Gamma$
\item If $\Gamma\to\Delta$ is injective then $\mathcal{F}_{\varphi}\left(\Gamma\to\Delta\right)$
is injective for every $\varphi\in\hom\left(\left(U,N_{U}\right),\left(X,W_{X}\right)\right)$.
(This is true generally the assumption that $\left(U,N_{U}\right)$
is a stencil space is unnecessary) 
\item $\core\circ\mathcal{F}_{\varphi}\left(\Gamma\to\Delta\right)=\mathcal{F}_{\varphi}\left(\Gamma\to\Delta\right)$ 
\end{enumerate}
\end{rem}

We can use a surjectivity problem $\left(\Gamma\to\Delta,\left(U,N_{U}\right)\right)$
from \cite{KOLODNER2020} as an injectivity problem. We say that an
injectivity problem resolves positively if all morphisms in $\mathscr{P}=\left\{ \text{Trimf}\circ\core\circ\mathcal{F}_{\varphi}\left(\Gamma\to\Delta\right)|\varphi\in\hom\left(\left(U,N_{U}\right),\left(X,W_{X}\right)\right)\right\} $
are injective. we distinguish three cases 
\begin{enumerate}
\item $\Gamma\to\Delta$ is not injective: clearly $\mathscr{P}$ resolves
negatively.
\item $\Gamma\to\Delta$ is injective and $\left(U,N_{U}\right)$ is a stencil
space of $\Delta$: following Remark \ref{rem:injective}, $\mathscr{P}$
resolves positively. 
\item $\Gamma\to\Delta$ is injective and $W\left(\Delta\right)\backslash N_{U}\neq\varnothing$:
in this case we cannot resolve $\mathscr{P}$ immediately. We call
this the ambiguous case.
\end{enumerate}
If $\mathscr{P}$ is of the ambiguous case we can split to five cases
using FGR. We examine the five new cases and then split again if necessary.
Because of Theorem $3.14$ in \cite{KOLODNER2020} every morphisms
$\varphi\in\hom\left(\left(U,N_{U}\right),\left(X,W_{X}\right)\right)$
either $\mathcal{F}_{\varphi}\left(\Gamma\to\Delta\right)$ isn't
injective or it ends up in a stencil case. Therefore we try to classify
all possible stencil cases that my arise in this process and determine
that they are all positive. In contrasted to the example in \cite{KOLODNER2020}
the graph $\Delta$ does not have stencil finitness therefore we end
this process differently. We notice the by conjugation we can assume
that $b$ is cyclically reduced so instead of $\left(\Gamma\to\Delta,\left(\left\{ a,b\right\} ,\emptyset\right)\right)$
we consider the problem $\left(\Gamma\to\Delta,\left(\left\{ a,b\right\} ,\left\{ b.b^{-1}\right\} \right)\right)$.
We preform a change of coordinates (see \cite{KOLODNER2020}). Let
$V=\left\{ a,b\right\} ,N_{V}=\left\{ b.b^{-1}\right\} $ , and
\[
\sigma:F_{\left\{ a,b\right\} }\to F_{\left\{ a,b\right\} },\quad\sigma\left(\alpha\right)=b,\quad\sigma\left(\beta\right)=aba^{-1}.
\]
We notice that $\left\langle b,aba^{-1}\right\rangle \leq\im\sigma$.
For any non-degenerate $\varphi\colon F_{\{a,b\}}\rightarrow F_{X}$,
the words $\varphi\left(b\right)=\varphi\circ\sigma\left(\alpha\right)$
and $\varphi\left(aba^{-1}\right)=\varphi\circ\sigma\left(\beta\right)$
are conjugate and $b$ is cyclically reduced, hence there exist reduced
words $\overline{y},\overline{u},\overline{v}\in F_{X}$ such that
$\varphi(b)=\overline{u}\cdot\overline{v}$, $\varphi(aba^{-1})=\overline{y}\cdot\overline{v}\cdot\overline{u}\cdot\overline{y}^{-1}$
(in particular, $\overline{v}\overline{u}$ and $\overline{u}\overline{v}$
are cyclically reduced). By non-degeneracy we can also assume $\overline{u}\neq1$,
and if $\overline{v}=1$ then $\overline{u}$ is cyclically reduced.
We perform a change of coordinates according to four possible cases,
with $(U_{i},N_{i})$, $\psi_{i}$ and $\sigma_{i}$ being:\medskip{}
\\
\hspace*{\fill}%
\begin{tabular}{|c|c|c|c|c|c|c|c|c|}
\hline 
\# & $\overline{y}$ & $\overline{v}$ & $U_{i}$ & $N_{i}$ & $\psi_{i}(\alpha),\psi_{i}(\beta)$ & $\sigma_{i}(a),\sigma_{i}(b)$ & $\Gamma_{i}$ & $\Delta_{i}$\tabularnewline
\hline 
\hline 
1 & $\negmedspace=\negmedspace1\negmedspace$ & $\negmedspace=\negmedspace1\negmedspace$ & $u$ & $u.u^{-1}$ & $u,u$ & $u,u$ & \includegraphics[scale=0.5]{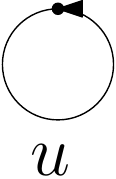} & \includegraphics[scale=0.5]{G2\lyxdot 1}\tabularnewline
\hline 
2 & $\negmedspace\neq\negmedspace1\negmedspace$ & $\negmedspace=\negmedspace1\negmedspace$ & $y,u$ & $y.u^{-1},u.y,u.u^{-1}$ & $u,yuy^{-1}$ & $y,u$ & \includegraphics[scale=0.5]{G2\lyxdot 1} & \includegraphics[scale=0.5]{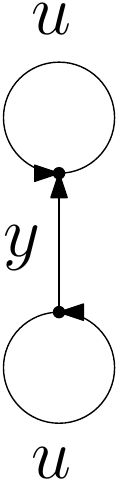}\tabularnewline
\hline 
3 & $\negmedspace=\negmedspace1\negmedspace$ & $\negmedspace\neq\negmedspace1\negmedspace$ & $v,u$ & $v.u^{-1},u.v^{-1}$ & $uv,vu$ & $u^{-1},uv$ & \includegraphics[scale=0.5]{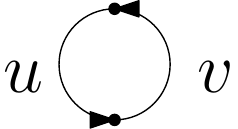} & \includegraphics[scale=0.5]{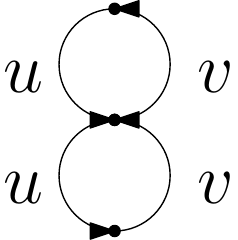}\tabularnewline
\hline 
4 & $\negmedspace\neq\negmedspace1\negmedspace$ & $\negmedspace\neq\negmedspace1\negmedspace$ & $v,u,y$ & $v.u^{-1},u.v^{-1},y.v^{-1},u.y$ & $uv,yvuy^{-1}$ & $yv,uv$ & \includegraphics[scale=0.5]{G4} & \includegraphics[scale=0.5]{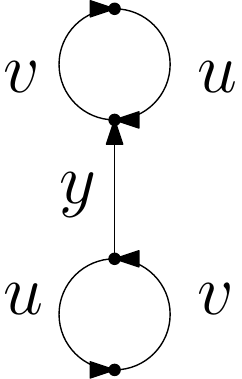}\tabularnewline
\hline 
\end{tabular} 

Case 3 is the problematic case it splits in to new cases indefinitely
this shows that $\left\langle b,aba^{-1}\right\rangle $ does not
have stencil finitness. In order to treat it we define two auxiliary
cases $x$ and $x'$. Cases $x$ and $x'$ include all the subcase
of case 3 but they are more general and includes many more cases that
are irrelevant to the original question. The advantage of using cases
$x$ and $x'$ is that they split into a finite set of ambiguous cases
in contrast to case 3 that splits into new ambiguous cases indefinitely.
In the table is an analysis of all the different cases.

\begin{tabular}{|c|c|c|c|c|c|c|c|}
\hline 
\# & FGR split & Homo. & $N_{i}$ & $\Gamma_{i}$ & $\Delta_{i}$ & $W\left(\Delta_{i}\right)\backslash N_{i}$ & Comment\tabularnewline
\hline 
2' &  &  & $\begin{array}{cc}
u.y, & y.u^{-1}\\
u^{-1}.y^{-1}, & u.u^{-1}
\end{array}$ & \includegraphics[scale=0.5]{G2\lyxdot 1} & \includegraphics[scale=0.5]{D2\lyxdot 1} & $u.y^{-1}$ & triangle rule+symmetry\tabularnewline
\hline 
2.1 & $u.y^{-1}.1$ & $\id$ & $\begin{array}{cc}
u.y, & y.u^{-1}\\
u.y,^{-1} & u.u^{-1}\\
u^{-1}.y^{-1}
\end{array}$ & \includegraphics[scale=0.5]{G2\lyxdot 1} & \includegraphics[scale=0.5]{D2\lyxdot 1} & $\emptyset$ & $\checkmark$\tabularnewline
\hline 
2.2 & $u.y^{-1}.2$ & $\begin{array}{c}
u\mapsto ut\phantom{^{-1}}\\
y\mapsto t^{-1}y
\end{array}$ & $\begin{array}{cc}
t.y, & y.u^{-1}\\
t.u^{-1}, & u.t^{-1}\\
u.y^{-1}, & t^{-1}.y^{-1}
\end{array}$ & \includegraphics[scale=0.5]{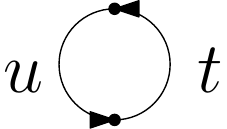} & \includegraphics[scale=0.5]{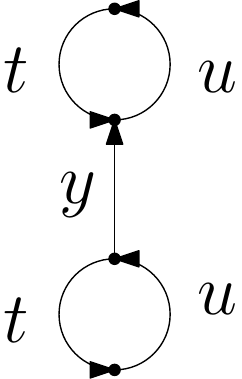} & $\emptyset$ & $\checkmark$\tabularnewline
\hline 
2.3 & $u.y^{-1}.3$ & $\begin{array}{c}
u\mapsto uy^{-1}\\
y\mapsto y\phantom{u^{-1}}
\end{array}$ & $\begin{array}{cc}
u.y, & y.u^{-1}\\
u^{-1}.y^{-1}, & y.y^{-1}
\end{array}$ & \includegraphics[scale=0.5]{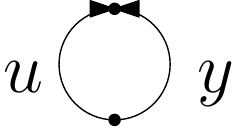} & \includegraphics[scale=0.5]{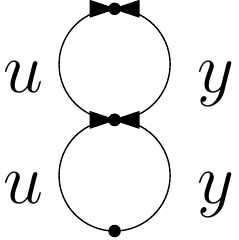} & $\begin{array}{c}
u.u^{-1}\\
u.y^{-1}
\end{array}$ & $\begin{array}{c}
\text{contained in case 3 via}\\
u\mapsto u\\
v\mapsto y^{-1}
\end{array}$\tabularnewline
\hline 
2.4 & $u.y^{-1}.4$ & $\begin{array}{c}
u\mapsto u\phantom{y^{-1}}\\
y\mapsto u^{-1}y
\end{array}$ & $\begin{array}{cc}
u.u^{-1}, & y.u^{-1}\\
u^{-1}.y,^{-1} & u.y
\end{array}$ & \includegraphics[scale=0.5]{G2\lyxdot 1} & \includegraphics[scale=0.5]{D2\lyxdot 1} & $u.y^{-1}$ & equivalent to case 2'\tabularnewline
\hline 
4' &  &  & $\begin{array}{cc}
v.u^{-1}, & u.v^{-1}\\
y.v^{-1}, & u.y\\
u^{-1}.y^{-1},
\end{array}$ & \includegraphics[scale=0.5]{G4} & \includegraphics[scale=0.5]{D4} & $v.y^{-1}$ & triangle rule+symmetry\tabularnewline
\hline 
4.1 & $v.y^{-1}.1$ & $\id$ & $\begin{array}{cc}
v.u^{-1}, & u.v^{-1}\\
y.v^{-1} & u.y\\
u^{-1}.y^{-1} & v.y^{-1}
\end{array}$ & \includegraphics[scale=0.5]{G4} & \includegraphics[scale=0.5]{D4} & $\emptyset$ & $\checkmark$\tabularnewline
\hline 
\end{tabular} 

\begin{tabular}{|c|c|c|c|c|c|c|c|}
\hline 
4.2 & $v.y^{-1}.2$ & $\begin{array}{c}
u\mapsto vt\phantom{^{-1}}\\
y\mapsto t^{-1}y
\end{array}$ & $\begin{array}{cc}
t.u^{-1}, & u.v^{-1}\\
y.v^{-1} & u^{-1}.y^{-1}\\
t^{-1}.y^{-1} & v.y^{-1}\\
v.t^{-1}
\end{array}$ & \includegraphics[scale=0.5]{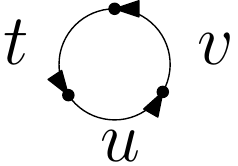} & \includegraphics[scale=0.5]{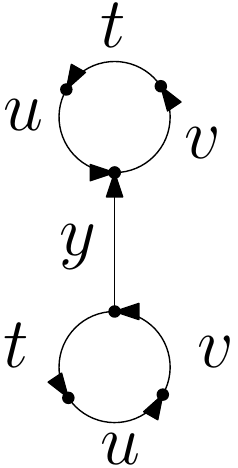} & $\emptyset$ & $\begin{array}{c}
\text{contained in case \ensuremath{2.2} via}\\
u\mapsto v\\
t\mapsto tu
\end{array}$\tabularnewline
\hline 
4.3 & $v.y^{-1}.3$ & $\begin{array}{c}
v\mapsto vy^{-1}\\
y\mapsto y\phantom{v^{-1}}
\end{array}$ & $\begin{array}{cc}
y^{-1}.u^{-1}, & y.v^{-1}\\
u.v^{-1}, & v.y\\
u.y
\end{array}$ & \includegraphics[scale=0.5]{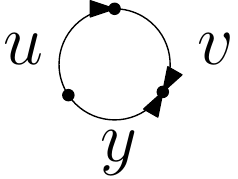} & \includegraphics[scale=0.5]{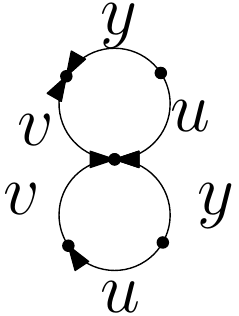} & $\begin{array}{c}
v.v^{-1}\\
v.u
\end{array}$ & $\begin{array}{c}
\text{contained in case 3 via}\\
u\mapsto y^{-1}u\\
v\mapsto v
\end{array}$\tabularnewline
\hline 
4.4 & $v.y^{-1}.4$ & $\begin{array}{c}
v\mapsto v\phantom{y^{-1}}\\
y\mapsto v^{-1}y
\end{array}$ & $\begin{array}{cc}
v.u^{-1}, & u.v^{-1}\\
y.v^{-1}, & u.y\\
v^{-1}.y^{-1}
\end{array}$ & \includegraphics[scale=0.5]{G4} & \includegraphics[scale=0.5]{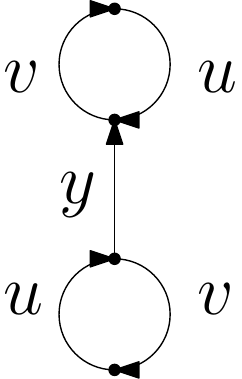} & $u.y^{-1}$ & $\begin{array}{c}
\text{contained in case 2 via}\\
u\mapsto vu\\
y\mapsto y
\end{array}$\tabularnewline
\hline 
3.1 & $u.v.1$ & $\id$ & $\begin{array}{cc}
v.u^{-1}, & u.v^{-1}\\
u.v
\end{array}$ & \includegraphics[scale=0.5]{G4} & \includegraphics[scale=0.5]{D3\lyxdot 1} & $\begin{array}{c}
u.u^{-1}\\
v.v^{-1}\\
v^{-1}.u^{-1}
\end{array}$ & ambiguous\tabularnewline
\hline 
3.2 & $u.v.2$ & $\begin{array}{c}
v\mapsto vt\\
u\mapsto ut
\end{array}$ & $\begin{array}{cc}
t.u^{-1}, & t.v^{-1}\\
v.t^{-1} & u.t^{-1}\\
u.v
\end{array}$ & \includegraphics[scale=0.5]{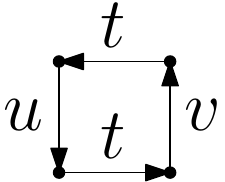} & \includegraphics[scale=0.5]{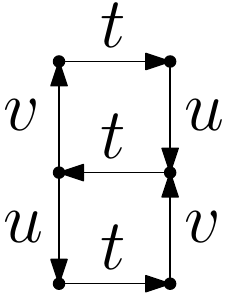} & $\begin{array}{c}
v^{-1}.u^{-1}\end{array}$ & $\begin{array}{c}
\text{contained in \ensuremath{x'} via}\\
x\mapsto t
\end{array}$\tabularnewline
\hline 
3.3 & $u.v.3$ & $\begin{array}{c}
v\mapsto v\phantom{u}\\
u\mapsto uv
\end{array}$ & $\begin{array}{cc}
v.u^{-1}, & v.v^{-1}\\
u.v^{-1}
\end{array}$ & \includegraphics[scale=0.5]{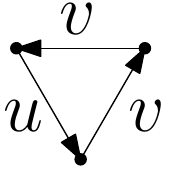} & \includegraphics[scale=0.5]{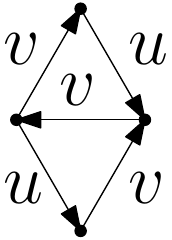} & $\begin{array}{c}
u.v\\
v^{-1}.u^{-1}
\end{array}$ & $\begin{array}{c}
\text{contained in \ensuremath{x} via}\\
x\mapsto v
\end{array}$\tabularnewline
\hline 
3.4 & $u.v.4$ & $\begin{array}{c}
v\mapsto vu\\
u\mapsto u\phantom{v}
\end{array}$ & $\begin{array}{cc}
u.u^{-1}, & u.v^{-1}\\
v.u^{-1}
\end{array}$ & \includegraphics[scale=0.5]{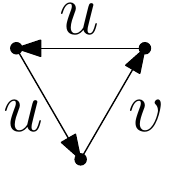} & \includegraphics[scale=0.5]{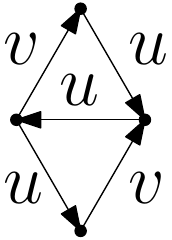} & $\begin{array}{c}
u.v\\
v^{-1}.u^{-1}
\end{array}$ & $\begin{array}{c}
\text{contained in \ensuremath{x} via}\\
x\mapsto u
\end{array}$\tabularnewline
\hline 
$x$ &  &  & $\begin{array}{cc}
v.u^{-1}, & u.v^{-1}\\
x.v^{-1}, & x.u^{-1}\\
u.x^{-1}, & v.x^{-1}
\end{array}$ & \includegraphics[scale=0.5]{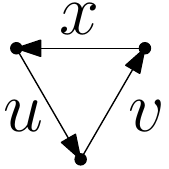} & \includegraphics[scale=0.5]{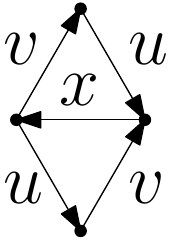} & $\begin{array}{c}
u.v\\
v^{-1}.u^{-1}
\end{array}$ & \tabularnewline
\hline 
$x'$ &  &  & $\begin{array}{cc}
t.u^{-1}, & t.v^{-1}\\
v.t^{-1}, & u.t^{-1}\\
x.u^{-1} & x.v^{-1}\\
x^{-1}.u, & x^{-1}.v\\
u.v
\end{array}$ & \includegraphics[scale=0.5]{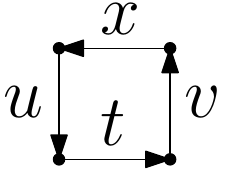} & \includegraphics[scale=0.5]{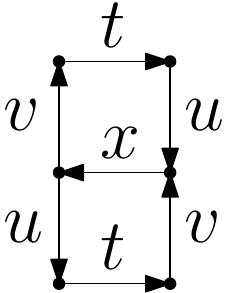} & $v^{-1}.u^{-1}$ & \tabularnewline
\hline 
\end{tabular} 

\begin{tabular}{|c|c|c|c|c|c|c|c|}
\hline 
$x.1$ & $u.v.1$ & $\id$ & $\begin{array}{cc}
v.u^{-1}, & u.v^{-1}\\
x.v^{-1}, & x.u^{-1}\\
u.x^{-1}, & v.x^{-1}\\
u.v
\end{array}$ & \includegraphics[scale=0.5]{G3\lyxdot 4g} & \includegraphics[scale=0.5]{D3\lyxdot 4g} & $u^{-1}.v^{-1}$ & ambiguous\tabularnewline
\hline 
$x.2$ & $u.v.2$ & $\begin{array}{c}
v\mapsto vt\\
u\mapsto ut
\end{array}$ & $\begin{array}{cc}
t.u^{-1}, & t.v^{-1}\\
x.v^{-1}, & x.u^{-1}\\
t.x^{-1}, & t.x^{-1}\\
v.t^{-1} & u.t^{-1}\\
u.v
\end{array}$ & \includegraphics[scale=0.5]{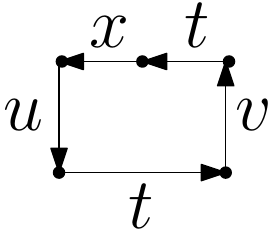} & \includegraphics[scale=0.5]{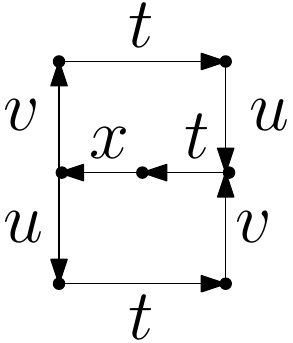} & $u^{-1}.v^{-1}$ & $\begin{array}{c}
\text{contained in \ensuremath{x'} via}\\
x\mapsto tx
\end{array}$\tabularnewline
\hline 
$x$.3 & $u.v.3$ & $\begin{array}{c}
v\mapsto vu\\
u\mapsto u\phantom{v}
\end{array}$ & $\begin{array}{cc}
u.u^{-1}, & u.v^{-1}\\
x.v^{-1}, & x.u^{-1}\\
u.x^{-1}, & v.u^{-1}
\end{array}$ & \includegraphics[scale=0.5]{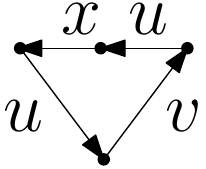} & \includegraphics[scale=0.5]{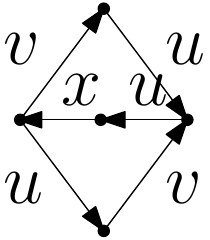} & $\begin{array}{c}
u.v\\
v^{-1}.u^{-1}
\end{array}$ & $\begin{array}{c}
\text{contained in \ensuremath{x} via}\\
x\mapsto ux
\end{array}$\tabularnewline
\hline 
$x.4$ & $u.v.4$ & $\begin{array}{c}
v\mapsto v\phantom{u}\\
u\mapsto uv
\end{array}$ & $\begin{array}{cc}
v.u^{-1}, & v.v^{-1}\\
x.v^{-1}, & x.u^{-1}\\
v.x^{-1}, & u.v^{-1}
\end{array}$ & \includegraphics[scale=0.5]{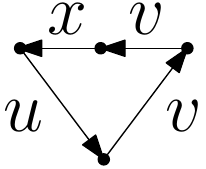} & \includegraphics[scale=0.5]{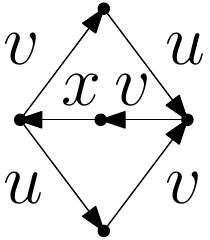} & $\begin{array}{c}
u.v\\
v^{-1}.u^{-1}
\end{array}$ & $\begin{array}{c}
\text{contained in \ensuremath{x} via}\\
x\mapsto vx
\end{array}$\tabularnewline
\hline 
$x.5$ & $u.v.5$ & $\begin{array}{c}
v\mapsto u\\
u\mapsto u
\end{array}$ & $\begin{array}{cc}
u.u^{-1},\\
u.x^{-1}, & x.u^{-1}
\end{array}$ & \includegraphics[scale=0.5]{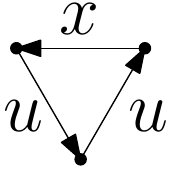} & \includegraphics[scale=0.5]{Dx\lyxdot 5} & $\emptyset$ & $\checkmark$\tabularnewline
\hline 
$x.1.1$ & $u^{-1}.v^{-1}.1$ & $\id$ & $\begin{array}{cc}
v.u^{-1}, & u.v^{-1}\\
x.v^{-1}, & x.u^{-1}\\
u.x^{-1}, & v.x^{-1}\\
u.v & u^{-1}.v^{-1}
\end{array}$ & \includegraphics[scale=0.5]{G3\lyxdot 4g} & \includegraphics[scale=0.5]{D3\lyxdot 4g} & $\emptyset$ & $\checkmark$\tabularnewline
\hline 
$x.1.2$ & $u^{-1}.v^{-1}.2$ & $\begin{array}{c}
v\mapsto tv\\
u\mapsto tu
\end{array}$ & $\begin{array}{cc}
v.t^{-1}, & u.t^{-1}\\
x.t^{-1}, & x.t^{-1}\\
u.x^{-1}, & v.x^{-1}\\
u.v, & t.u^{-1}\\
t.v^{-1}. & u^{-1}.v^{-1}
\end{array}$ & \includegraphics[scale=0.5]{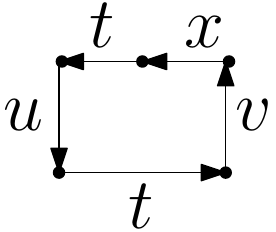} & \includegraphics[scale=0.5]{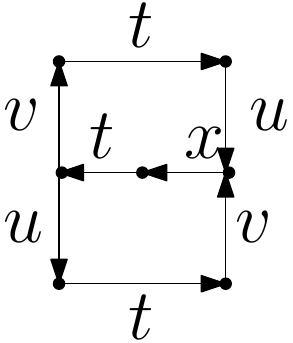} & $\emptyset$ & $\begin{array}{c}
\text{contained in \ensuremath{x}'.1 via}\\
x\mapsto xt
\end{array}$\tabularnewline
\hline 
$x.1.3$ & $u^{-1}.v^{-1}.3$ & $\begin{array}{c}
v\mapsto uv\\
u\mapsto u\phantom{v}
\end{array}$ & $\begin{array}{cc}
v.u^{-1}, & u.u^{-1}\\
x.u^{-1}, & u.v^{-1}\\
u.x^{-1}, & v.x^{-1}\\
u.v
\end{array}$ & \includegraphics[scale=0.5]{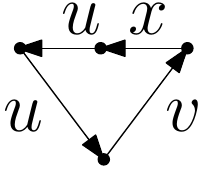} & \includegraphics[scale=0.5]{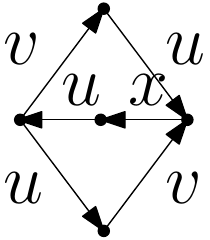} & $u^{-1}.v^{-1}$ & $\begin{array}{c}
\text{contained in \ensuremath{x}.1 via}\\
x\mapsto xu
\end{array}$\tabularnewline
\hline 
$x.1.4$ & $u^{-1}.v^{-1}.4$ & $\begin{array}{c}
v\mapsto v\phantom{u}\\
u\mapsto vu
\end{array}$ & $\begin{array}{cc}
v.v^{-1}, & u.v^{-1}\\
x.v^{-1}, & u.x^{-1}\\
v.x^{-1}, & u.v\\
v.u^{-1}
\end{array}$ &  &  & $u^{-1}.v^{-1}$ & $\begin{array}{c}
\text{contained in \ensuremath{x}.1 via}\\
x\mapsto xv
\end{array}$\tabularnewline
\hline 
\end{tabular} 

\begin{tabular}{|c|c|c|c|c|c|c|c|}
\hline 
$x'$.1 & $v^{-1}.u^{-1}$.1 & $\id$ & $\begin{array}{cc}
t.u^{-1}, & t.v^{-1}\\
v.t^{-1} & u.t^{-1}\\
u.v & x.v^{-1}\\
x.u^{-1} & u.x^{-1}\\
v.x^{-1} & v^{-1}.u^{-1}
\end{array}$ & \includegraphics[scale=0.5]{Gx_} & \includegraphics[scale=0.5]{Dx_} & $v^{-1}.u^{-1}$ & $\checkmark$\tabularnewline
\hline 
$x'$.2 & $v^{-1}.u^{-1}$.2 & $\begin{array}{c}
v\mapsto sv\\
u\mapsto su
\end{array}$ & $\begin{array}{cc}
t.s^{-1}, & t.s^{-1}\\
v.t^{-1} & u.t^{-1}\\
u.v & x.s^{-1}\\
x.s^{-1} & u.x^{-1}\\
v.x^{-1} & v^{-1}.u^{-1}\\
s.v^{-1} & s.u^{-1}
\end{array}$ & \includegraphics[scale=0.5]{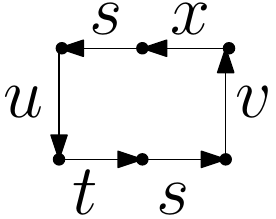} & \includegraphics[scale=0.5]{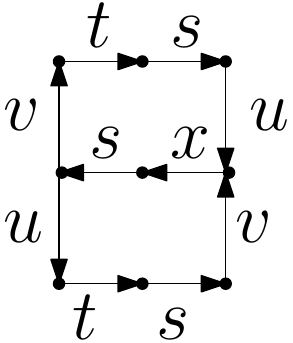} & $\emptyset$ & $\begin{array}{c}
\text{contained in \ensuremath{x}'.1 via}\\
x\mapsto xs\\
u\mapsto ts
\end{array}$\tabularnewline
\hline 
$x'$.3 & $v^{-1}.u^{-1}.3$ & $\begin{array}{c}
v\mapsto uv\\
u\mapsto u\phantom{v}
\end{array}$ & $\begin{array}{cc}
t.u^{-1}, & u.v^{-1}\\
v.t^{-1} & u.t^{-1}\\
u.v & x.u^{-1}\\
v.x^{-1} & u.x^{-1}
\end{array}$ & \includegraphics[scale=0.5]{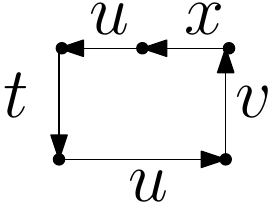} & \includegraphics[scale=0.5]{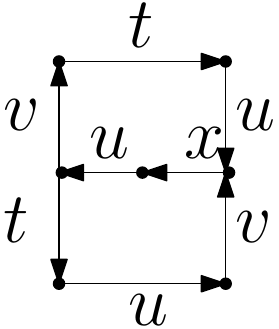} & $v^{-1}.u^{-1}$ & $\begin{array}{c}
\text{contained in \ensuremath{x}.1 via}\\
x\mapsto xu\\
u\mapsto tu
\end{array}$\tabularnewline
\hline 
$x'.4$ & $v^{-1}.u^{-1}.4$ & $\begin{array}{c}
v\mapsto v\phantom{u}\\
u\mapsto vu
\end{array}$ & $\begin{array}{cc}
t.v^{-1}, & v.u^{-1}\\
v.t^{-1}, & u.t^{-1}\\
x.v^{-1}, & u.v\\
x^{-1}.u, & x^{-1}.v
\end{array}$ & \includegraphics[scale=0.5]{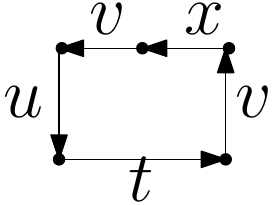} & \includegraphics[scale=0.5]{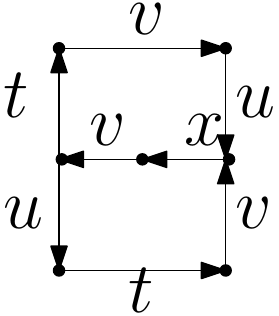} & $v^{-1}.u^{-1}$ & $\begin{array}{c}
\text{contained in \ensuremath{x}.1 via}\\
x\mapsto xv\\
v\mapsto tv
\end{array}$\tabularnewline
\hline 
\end{tabular}

\begin{tabular}{|c|c|c|c|c|c|c|c|}
\hline 
3.1.1 & $v^{-1}.u^{-1}.1$ & $\id$ & $\begin{array}{cc}
v.u^{-1}, & u.v^{-1}\\
v.v^{-1} & u.v
\end{array}$ & \includegraphics[scale=0.5]{G4} & \includegraphics[scale=0.5]{D3\lyxdot 1} & $\begin{array}{c}
u.u^{-1}\\
v.v^{-1}
\end{array}$ & ambiguous\tabularnewline
\hline 
3.1.1.2 & $v.v^{-1}.2$ & $\begin{array}{c}
v\mapsto t^{-1}vt\\
u\mapsto u\phantom{tt^{-1}}
\end{array}$ & $\begin{array}{cc}
t.u^{-1}, & u.t\\
v.v^{-1}, & v.t^{-1}\\
t^{-1}.v
\end{array}$ & \includegraphics[scale=0.5]{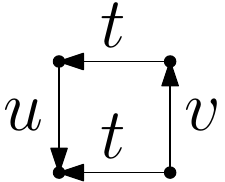} & \includegraphics[scale=0.5]{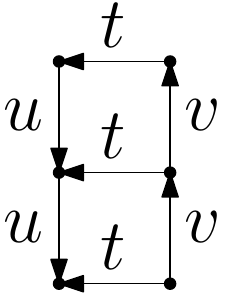} & $u.u^{-1}$ & ambiguous\tabularnewline
\hline 
3.1.1.2.1 & $u.u^{-1}.1$ & $\id$ & $\begin{array}{cc}
t.u^{-1}, & u.t\\
v.v^{-1}, & v.t^{-1}\\
t^{-1}.v, & u.u^{-1}
\end{array}$ &  &  & $\emptyset$ & $\checkmark$\tabularnewline
\hline 
3.1.1.2.2 & $u.u^{-1}.2$ & $\begin{array}{c}
v\mapsto v\phantom{ss^{-1}}\\
u\mapsto s^{-1}us
\end{array}$ & $\begin{array}{cc}
t.s, & v.t^{-1}\\
v.v^{-1}, & t^{-1}.v\\
u.s^{-1}, & u^{-1}.s^{-1}\\
u.u^{-1}
\end{array}$ & \includegraphics[scale=0.5]{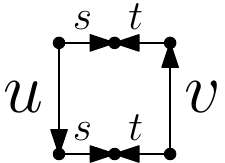} & \includegraphics[scale=0.5]{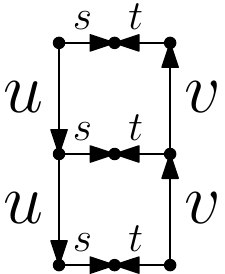} & $\emptyset$ & $\checkmark$\tabularnewline
\hline 
\end{tabular}

\bibliographystyle{plain}
\bibliography{mybibn}
 
\end{document}